\documentclass[11pt]{article}
\usepackage{float}
\usepackage{amsmath}
\usepackage{amssymb}
\usepackage{amsthm}
\usepackage{geometry}
\usepackage{booktabs}
\usepackage[colorlinks=true, citecolor=blue, linkcolor=blue, urlcolor=blue]{hyperref}
\usepackage[utf8]{inputenc}
\usepackage{xurl}
\newtheorem{theorem}{Theorem}[section]
\newtheorem{lemma}[theorem]{Lemma}
\newtheorem{proposition}[theorem]{Proposition}
\newtheorem{definition}[theorem]{Definition}
\newtheorem{remark}[theorem]{Remark}

\newtheorem*{foxconjecture}{Fox Trapezoidal Conjecture} 

\title{Spectral Factorization and Hypergeometric Representations of the Alexander Polynomials of $Th(4,2n+1)$
}
\author{Suman Saurabh}
\date{\vspace{-4ex}}

\begin{document}

\maketitle

\begin{abstract}
We study the Alexander polynomials of the 4-strand Turk's head knots $Th(4,2n+1)$, defined as the closures of the braid $(\sigma_1\sigma_2^{-1}\sigma_3)^{2n+1}$. Using the reduced Burau representation, we derive an annihilating recurrence of order at most 8 and a rational generating function for the resulting polynomial sequence. By executing a multivariable resultant elimination over the reciprocal constraint, we obtain an exact factorization of the normalized Alexander polynomial in terms of Chebyshev polynomials. This factorization produces a binomial convolution formula for an associated coefficient sequence and a representation by a terminating ${}_4F_3$ hypergeometric series. We evaluate the continuous approximation of this representation using the saddle-point method, demonstrating negative curvature in the asymptotic main term. Finally, we describe analytic obstructions to extracting global discrete error bounds via this method, leaving the formal proof of Fox's Trapezoidal Conjecture for this family open.
\end{abstract}

\vspace{0.5cm}

\section{Introduction}

The Alexander polynomial is a standard invariant in classical knot theory \cite{Kawauchi1996, Lickorish1997, Rolfsen1976}. Understanding the coefficient structure of Alexander polynomials across infinite knot families remains a central problem in geometric topology. Of particular interest is the distribution of their coefficients, governed by Fox's Trapezoidal Conjecture:

\begin{foxconjecture}
For an alternating knot, the absolute values of the coefficients of its Alexander polynomial form a trapezoidal sequence. A stronger variant posits that the non-plateau portions of the sequence are log-concave, satisfying the inequality $a_i^2 > |a_{i-1}| \cdot |a_{i+1}|$.
\end{foxconjecture}

While Murasugi \cite{Murasugi1958} established fundamental algebraic properties for alternating knots, structural proofs of the trapezoidal condition remain an open problem for general alternating families. However, recent advancements have established the conjecture for specific classes, such as diagrammatic Murasugi sums and twist-concentrated links, using topological and diagrammatic bounds \cite{azarpendar2024foxstrapezoidalconjecture}. For infinite braid families where the crossing number grows arbitrarily large, such as Th(4,q), direct algebraic factorizations and asymptotic methods provide a necessary complementary approach.

Evaluating the determinant of the Burau representation \cite{Burau1936} for higher-strand braids as the crossing number $q \to \infty$ introduces significant computational complexity. While divisibility and structural properties of the polynomial invariants for $Th(p,q)$ have been investigated using cyclotomic block decompositions \cite{Takemura2018}, explicit closed-form spectral factorizations have primarily been limited to the 3-strand case (see \cite{AlSukaiti2023} and the recent survey by \cite{diprisa2024turksheadknotslinks}). The 4-strand family $Th(4,q)$ is substantially more complicated than the 3-strand case because the associated Burau characteristic polynomial is cubic rather than quadratic. 

In this paper, we present three structural results for the family $Th(4,2n+1)$:
\begin{enumerate}
    \item We bypass direct determinant evaluation by constructing an annihilating operator for the Burau sequence. Applying an even-step transformation isolates the odd-$q$ polynomial sequence, yielding a rational generating function over $\mathbb{Z}[t, t^{-1}]$. 
    \item We obtain a Chebyshev factorization of a normalized Alexander polynomial through a resultant computation. This factorization leads to a trigonometric product representation, an explicit binomial convolution formula for the coefficients, and a terminating hypergeometric representation.
    \item We analyze the continuous approximation of the hypergeometric representation via the saddle-point method. This calculation demonstrates negative curvature in the continuous main term. We also document obstructions to bounding the discrete error terms, showing that standard global Taylor remainders are insufficient to prove discrete log-concavity for finite $n$.
\end{enumerate}

While focused specifically on the invariants of $Th(4,2n+1)$, the spectral factorization technique established herein—evaluating the Burau representation over roots of unity and reducing via symmetric resultant polynomials—provides an explicit algebraic template theoretically adaptable to other alternating braid closures.

The remainder of this paper is organized as follows. In Section~\ref{sec:BraidRepr&Recurr}, we utilize the reduced Burau representation to construct an annihilating operator, establishing a rational generating function for the Alexander polynomials of $Th(4, 2n+1)$. Section~\ref{sec:SpecFactrz&Chebyshev} defines a reciprocal normalization and executes a resultant computation over roots of unity to factor the sequence algebraically via Chebyshev polynomials. In Section~\ref{sec:HyperGeomRed}, we expand this factorization into a discrete binomial convolution, obtaining an explicit sequence representation as a terminating ${}_4F_3$ generalized hypergeometric series. Finally, Section~\ref{sec:AsympApprox} evaluates the continuous limit of this hypergeometric sequence using the saddle-point method to demonstrate asymptotic log-concavity, concluding with an analysis of the analytic obstructions that inhibit the extraction of global discrete remainder bounds. Explicit polynomial coefficients and algebraic resultant expansions are recorded in the Appendices.

\section{The Braid Representation and Recurrence} \label{sec:BraidRepr&Recurr}

Let $B_4$ denote the braid group on 4 strands. We consider the infinite family of knots generated by the closure of the mixed-sign braid word $\beta_q = (\sigma_1\sigma_2^{-1}\sigma_3)^q$. We restrict the index to odd integers $q = 2n+1$ for $n \ge 1$. This ensures the coprimality condition $\gcd(4, q) = 1$, which guarantees the geometric closure forms a single-component knot. 

We employ the reduced Burau representation $B_4 \to GL_3(\mathbb{Z}[t, t^{-1}])$. The standard generators evaluate to:
\begin{equation*}
\rho(\sigma_1) = \begin{pmatrix} -t & 1 & 0 \\ 0 & 1 & 0 \\ 0 & 0 & 1 \end{pmatrix}, \quad
\rho(\sigma_2)^{-1} = \begin{pmatrix} 1 & 0 & 0 \\ 1 & -t^{-1} & t^{-1} \\ 0 & 0 & 1 \end{pmatrix}, \quad
\rho(\sigma_3) = \begin{pmatrix} 1 & 0 & 0 \\ 0 & 1 & 0 \\ 0 & t & -t \end{pmatrix}.
\end{equation*}

\subsection{Characteristic Polynomial and Trace Formula}

Let $M(t) = \rho(\sigma_1\sigma_2^{-1}\sigma_3)$ denote the fundamental matrix block. 

\begin{lemma}\label{lem:char_poly}
The characteristic polynomial $\chi_M(E,t) = \det(EI - M(t))$ is given by:
\begin{equation}
\chi_M(E,t) = E^3 + (2t + t^{-1} - 2)E^2 + (t^2 - 2t + 2)E + t.
\end{equation}
\end{lemma}

\begin{proof}
Matrix multiplication of the constituent generators yields the base matrix:
\begin{equation*}
M(t) = \begin{pmatrix} 1-t & 1-t^{-1} & -1 \\ 1 & 1-t^{-1} & -1 \\ 0 & t & -t \end{pmatrix}.
\end{equation*}
The characteristic polynomial expands via its principal invariants as $E^3 - \operatorname{tr}(M)E^2 + c_1 E - \det(M)$, where $c_1$ represents the sum of the principal $2 \times 2$ minors. We explicitly compute these invariants:
\begin{enumerate}
    \item \textbf{Trace:} $\operatorname{tr}(M) = (1-t) + (1-t^{-1}) - t = 2 - 2t - t^{-1}$.
    \item \textbf{Principal Minors ($c_1$):} The sum of the principal $2 \times 2$ minors evaluates to:
    \begin{align*}
    c_1 &= \det\begin{pmatrix} 1-t^{-1} & -1 \\ t & -t \end{pmatrix} + \det\begin{pmatrix} 1-t & -1 \\ 0 & -t \end{pmatrix} + \det\begin{pmatrix} 1-t & 1-t^{-1} \\ 1 & 1-t^{-1} \end{pmatrix} \\
    &= \big( -t(1-t^{-1}) + t \big) + \big( -t(1-t) \big) + \big( (1-t)(1-t^{-1}) - (1-t^{-1}) \big) \\
    &= (1) + (t^2-t) + (1 - t^{-1} - t + 1 - 1 + t^{-1}) \\
    &= 1 + t^2 - t - t + 1 = t^2 - 2t + 2.
    \end{align*}
    \item \textbf{Determinant:} Expanding along the first column shows $\det(M) = -t$.
\end{enumerate}
Substitution yields the stated polynomial.
\end{proof}

The Alexander polynomial of the closed braid is given by Burau's determinant identity $\Delta_q(t) \doteq \frac{1-t}{1-t^4} \det(I - M(t)^q)$ \cite{Birman1974, Burau1936}. Factoring the denominator gives:
\begin{equation*}
\Delta_q(t) \doteq \frac{\det(I - M(t)^q)}{1+t+t^2+t^3}.
\end{equation*}
Let $F_q(t) = \det(I - M(t)^q)$. We decompose this determinant into linear traces. 

\begin{lemma}[Exterior Trace Decomposition]\label{lem:trace}
The determinant sequence decomposes as:
\begin{equation}
F_q(t) = 1 - \operatorname{tr}(M^q) + \operatorname{tr}\big((\wedge^2 M)^q\big) - (-t)^q.
\end{equation}
\end{lemma}

\begin{proof}
Let $\lambda_1, \lambda_2, \lambda_3$ be the eigenvalues of $M(t)$. Expanding the determinant definition over the algebraic closure gives:
\begin{align*}
\det(I - M^q) &= (1-\lambda_1^q)(1-\lambda_2^q)(1-\lambda_3^q) \\
&= 1 - (\lambda_1^q + \lambda_2^q + \lambda_3^q) + (\lambda_1^q\lambda_2^q + \lambda_2^q\lambda_3^q + \lambda_3^q\lambda_1^q) - (\lambda_1^q\lambda_2^q\lambda_3^q).
\end{align*}
By Newton's identities and the properties of exterior algebras, we substitute the following invariants:
\begin{align*}
\sum \lambda_i^q &= \operatorname{tr}(M^q), \\
\sum \lambda_i^q\lambda_j^q &= \operatorname{tr}\big((\wedge^2 M)^q\big), \\
\prod \lambda_i^q &= \det(M^q) = (-t)^q.
\end{align*}
Substituting these into the expanded product gives the decomposition.
\end{proof}

\subsection{Annihilating Operator and Generating Function}

\begin{lemma}[Annihilating Operator]\label{lem:annihilator}
Let $E$ be the forward shift operator such that $E(F_q) = F_{q+1}$. The sequence $F_q(t)$ is annihilated by the operator $L(E) = (E-1)(E+t)P_1(E)P_2(E)$, where:
\begin{align*}
P_1(E) &= E^3 + (2t + t^{-1} - 2)E^2 + (t^2 - 2t + 2)E + t, \\
P_2(E) &= E^3 - (t^2 - 2t + 2)E^2 + (2t^2 - 2t + 1)E - t^2.
\end{align*}
\end{lemma}

\begin{proof}
By Lemma~\ref{lem:trace}, the sequence $F_q(t)$ is a linear combination of a constant sequence, a geometric sequence, and two trace sequences. The set of sequences annihilated by a linear shift operator with constant coefficients forms a vector space. An annihilating operator for the sum is given by the product of the individual annihilators for these constituent sequences:
\begin{enumerate}
    \item The constant sequence $1$ is annihilated by $(E-1)$.
    \item The geometric sequence $(-t)^q$ is annihilated by $(E+t)$.
    \item It is a standard property of holonomic sequences that if a matrix $M$ has characteristic polynomial $\chi_M(E,t)=0$, the sequence of its traces $v_q = \operatorname{tr}(M^q)$ satisfies the linear recurrence defined by the coefficients of $\chi_M(E,t)$. Thus, $P_1(E) = \chi_M(E,t)$ annihilates the trace sequence.
    \item The sequence $\operatorname{tr}\big((\wedge^2 M)^q\big)$ is similarly annihilated by the characteristic polynomial of the exterior square $\wedge^2 M$. Let the characteristic polynomial of $M$ be parameterized as $E^3 - aE^2 + bE - c$. The eigenvalues of $\wedge^2 M$ are the pairwise products $\lambda_1\lambda_2, \lambda_2\lambda_3, \lambda_3\lambda_1$. Utilizing elementary symmetric polynomials, we obtain:
    \begin{align*}
    \sum \lambda_i\lambda_j &= b, \\
    \sum (\lambda_i\lambda_j)(\lambda_j\lambda_k) &= \lambda_1\lambda_2\lambda_3(\lambda_1+\lambda_2+\lambda_3) = ca, \\
    \prod \lambda_i\lambda_j &= (\lambda_1\lambda_2\lambda_3)^2 = c^2.
    \end{align*}
    Thus, the characteristic polynomial of $\wedge^2 M$ evaluates to $E^3 - bE^2 + acE - c^2$. Substituting $a = 2 - 2t - t^{-1}$, $b = t^2 - 2t + 2$, and $c = -t$ from Lemma~\ref{lem:char_poly} yields the operator $P_2(E)$. Consequently, $P_2(E)$ annihilates the sequence $\operatorname{tr}\big((\wedge^2 M)^q\big)$.
\end{enumerate}
Since these operators commute, their product $L(E) = (E-1)(E+t)P_1(E)P_2(E)$ annihilates the direct sum of the kernels of its factors. Therefore, $L(E)$ annihilates the combined sequence $F_q(t)$.
\end{proof}

\begin{lemma}[Even-Step Transformation]\label{lem:even_step}
Let $H_n(t) = F_{2n+1}(t)$ denote the downsampled odd-indexed subsequence. $H_n(t)$ satisfies a homogeneous linear recurrence of order at most 8.
\end{lemma}

\begin{proof}
The operator $L(E)$ has degree 8 in $E$. We construct the companion operator $R(E) = L(E)L(-E)$. Any polynomial operator $L(E)$ can be decomposed into its even and odd components, $L(E) = A(E^2) + E B(E^2)$. Multiplying by the radially reflected counterpart eliminates all odd powers of the shift operator via a difference of squares:
\begin{equation*}
R(E) = \big( A(E^2) + E B(E^2) \big)\big( A(E^2) - E B(E^2) \big) = A(E^2)^2 - E^2 B(E^2)^2.
\end{equation*}
Because $R(E)$ contains even powers of $E$, it can be written as a polynomial $Q(E^2)$ of degree 8 in $E^2$. Since $L(E)$ annihilates $F_q(t)$, and polynomial operators with constant coefficients commute, $R(E)$ annihilates $F_q(t)$, meaning $Q(E^2)$ applied to $F_q(t)$ evaluates to zero.

Because $Q(E^2)$ annihilates $F_q(t)$, we have $\sum_{j=0}^8 c_j E^{2j} F_q(t) = 0$, giving $\sum_{j=0}^8 c_j F_{q+2j}(t) = 0$. Restricting the sequence to the odd lattice $q = 2n+1$ yields $\sum_{j=0}^8 c_j F_{2(n+j)+1}(t) = 0$. By the definition of the downsampled sequence $H_k(t) = F_{2k+1}(t)$, this equates to $\sum_{j=0}^8 c_j H_{n+j}(t) = 0$. This defines a homogeneous linear recurrence of order at most 8 for $H_n(t)$, governed by a new forward shift operator $E'$ on the downsampled index $n$, defined by $E'(H_n) = H_{n+1}$.
 
\end{proof}

\begin{theorem}[Recurrence and Generating Function]\label{thm:macro_gf}
Let $P_n(t) = \Delta_{2n+1}(t)$ denote the Alexander knot polynomial sequence for $n \ge 1$. The sequence $P_n(t)$ satisfies a linear recurrence of order at most 8 and admits a rational generating function.
\end{theorem}

\begin{proof}
Because the denominator $(1+t+t^2+t^3)$ of the Alexander polynomial identity $\Delta_{2n+1}(t) \doteq \frac{F_{2n+1}(t)}{1+t+t^2+t^3}$ is independent of the sequence index $n$, the linear operator $Q(E')$ from Lemma~\ref{lem:even_step} that annihilates the numerator simultaneously annihilates the quotient $P_n(t)$. Thus, $P_n(t)$ satisfies the linear recurrence:
\begin{equation*}
\sum_{j=0}^8 c_j(t) P_{m+j}(t) = 0 \quad \text{for all } m \ge 0.
\end{equation*}
Let $G(x,t) = \sum_{m=0}^\infty P_m(t) x^m$ be the power series. Multiplying the recurrence by $x^{m+8}$ and summing over all $m \ge 0$ yields:
\begin{equation*}
\sum_{m=0}^\infty \sum_{j=0}^8 c_j(t) P_{m+j}(t) x^{m+8} = 0.
\end{equation*}
Applying the index shift $k = m+j$, we can rewrite the summation over the bounds of $k$:
\begin{equation*}
\sum_{j=0}^8 c_j(t) x^{8-j} \sum_{k=j}^\infty P_k(t) x^k = 0.
\end{equation*}
Substituting the identity $\sum_{k=j}^\infty P_k(t) x^k = G(x,t) - \sum_{k=0}^{j-1} P_k(t) x^k$ isolates the generating function:
\begin{equation*}
\left( \sum_{j=0}^8 c_j(t) x^{8-j} \right) G(x,t) - \sum_{j=1}^8 c_j(t) x^{8-j} \left( \sum_{k=0}^{j-1} P_k(t) x^k \right) = 0.
\end{equation*}
We define the characteristic denominator polynomial 
\begin{equation*}
    D(x,t) = \sum_{j=0}^8 c_j(t) x^{8-j} = \sum_{j=0}^8 c_{8-j}(t) x^j
\end{equation*}
The remaining double summation on the right-hand side contains only finite initial values $P_k(t)$ for $k < 8$, defining a finite polynomial numerator $N(x,t)$ of degree $\le 7$. Dividing by $D(x,t)$ proves the generating function is rational. Thus, $P_n(t)$ admits a rational generating function of the form $G(x,t) = \frac{N(x,t)}{D(x,t)}$. Theorem~\ref{thm:macro_gf} establishes the existence of this rational form; Appendix \ref{app:gf} records the computer-generated polynomials $N(x,t)$ and $D(x,t)$.
\end{proof}

\section{Spectral Factorization and Chebyshev Structure} \label{sec:SpecFactrz&Chebyshev}

\subsection{Matrix Factorization over Roots of Unity}

\begin{lemma}[Determinant Factorization]\label{lem:det_fact}
Let $q = 2n+1$. Over the $q$-th roots of unity $\zeta^j = \exp\left(\frac{2\pi i j}{q}\right)$, the Burau determinant expands as:
\begin{equation}
\det(I - M^q) = \prod_{j=0}^{2n} \det(I - \zeta^{-j} M).
\end{equation}
\end{lemma}

\begin{proof}
The polynomial $x^q - 1$ factors as $\prod_{j=0}^{q-1} (x - \zeta^j)$. Evaluating this identity at the matrix $M$ yields the matrix factorization $I - M^q = \prod_{j=0}^{q-1} (I - \zeta^{-j} M)$. As the factors are polynomials in the same matrix $M$, they commute. Taking the determinant of both sides yields the product representation.
\end{proof}

\begin{lemma}[Trivial Root Evaluation]\label{lem:trivial_root}
The trivial root $j=0$ evaluates to:
\[ \det(I - M(t)) = t^{-1}(1 + t + t^2 + t^3). \]
\end{lemma}
\begin{proof}
By direct evaluation of the identity offset matrix:
\begin{equation*}
I - M(t) = \begin{pmatrix} t & t^{-1}-1 & 1 \\ -1 & t^{-1} & 1 \\ 0 & -t & 1+t \end{pmatrix}.
\end{equation*}
Expanding the determinant along the first column yields:
\begin{align*}
\det(I-M(t)) &= t \big( t^{-1}(1+t) - (-t) \big) - (-1) \big( (t^{-1}-1)(1+t) - (-t) \big) + 0 \\
&= t \big( t^{-1} + 1 + t \big) + \big( t^{-1} + 1 - 1 - t + t \big) \\
&= \big( 1 + t + t^2 \big) + t^{-1} = t^{-1}(1 + t + t^2 + t^3).
\end{align*}
\end{proof}

\subsection{Reciprocal Normalization}

\begin{definition}[Normalized Polynomial Representation]\label{def:recip_poly}
The Burau determinant identity defines the Alexander polynomial up to a unit $\pm t^k$. We fix this ambiguity by defining the normalized ordinary polynomial $A_{2n+1}(z)$ to be the polynomial resulting from the unnormalized determinant expansion evaluated at $t=-z$ and scaled by $z^{2n+1}$:
\begin{equation}
A_{2n+1}(z) = -z^{2n+1} \frac{\det(I - M(-z)^{2n+1})}{1-z+z^2-z^3}.
\end{equation}
This selects a unique representative of the Alexander equivalence class. We define its corresponding symmetric Laurent representative as $\nabla_{2n+1}(z) = z^{-3n} A_{2n+1}(z)$.
\end{definition}

\subsection{Resultant Elimination}

Evaluating the characteristic polynomial $\chi_M(E,t)$ from Lemma~\ref{lem:char_poly} at $t=-z$, and replacing the indeterminate $E$ with $x$ to distinguish it from the shift operator, yields:
\begin{equation}
\chi_M(x, -z) = x^3 - (2z + z^{-1} + 2)x^2 + (z^2 + 2z + 2)x - z.
\end{equation}

\begin{lemma}[Reciprocity Relation]\label{lem:reciprocity}
The characteristic polynomial satisfies the reciprocity identity:
\[
\chi_M(x^{-1},-z) = -zx^{-3}\chi_M(x,-z^{-1}).
\]
\end{lemma}
\begin{proof}
Evaluating $\chi_M(x, -z^{-1})$ at $x$ and multiplying the entire expression by $-zx^{-3}$ implies:
\begin{align*}
-zx^{-3} \chi_M(x, -z^{-1}) &= -zx^{-3} \left( x^3 - (2z^{-1}+z+2)x^2 + (z^{-2}+2z^{-1}+2)x - z^{-1} \right) \\
&= -z + (2+z^2+2z)x^{-1} - (z^{-1}+2+2z)x^{-2} + x^{-3} \\
&= \chi_M(x^{-1}, -z).
\end{align*}
\end{proof}

\begin{lemma}[Resultant Polynomial]\label{lem:resultant}
Let $v = x + x^{-1}$ denote the reciprocal spectral parameter. Clearing the Laurent denominator by defining $\widetilde{\chi}_M(x, w) = w \chi_M(x, -w)$, the resultant polynomial over the constraint $w^2-uw+1=0$ evaluates to:
\begin{equation}
R(x, u) = \operatorname{Res}_w\big(\widetilde{\chi}_M(x, w), w^2-uw+1\big) = -x^3 (u-v)(u-v+2)^2.
\end{equation}
\end{lemma}

\begin{proof}
Let $z$ and $z^{-1}$ be the roots of $w^2 - uw + 1 = 0$. By definition, the resultant evaluates to the product over these roots:
\begin{equation*}
R(x, u) = \chi_M(x, -z) \chi_M(x, -z^{-1}).
\end{equation*}
Because the roots map identically under the involution $z \mapsto z^{-1}$, the product expands into a polynomial in $x$ whose coefficients are symmetric Laurent polynomials in $z$. Substituting the fundamental invariant $u = z+z^{-1}$ translates the product into a polynomial in $x$ and $u$. Factoring out $x^3$ isolates the reciprocal variable $v = x+x^{-1}$, reducing the expression to a cubic in $v$. The full algebraic expansion and coefficient reduction verifying the factorization $-x^3(u-v)(u-v+2)^2$ are detailed in Appendix \ref{app:resultant}.
\end{proof}

\subsection{Chebyshev Factorization}

The determinant factorization (Lemma~\ref{lem:det_fact}) evaluates the characteristic polynomial at the nontrivial $(2n+1)$-th roots of unity, yielding Chebyshev nodes $v_r = 2\cos\left(\frac{2\pi r}{2n+1}\right)$.

\begin{lemma}[Chebyshev Roots]\label{lem:cheb_roots}
The characteristic Chebyshev combination 
\begin{equation}
F_n(u) = U_n\left(\frac{u}{2}\right) + U_{n-1}\left(\frac{u}{2}\right)
\end{equation}
is monic of degree $n$ and factors as $F_n(u) = \prod_{r=1}^n (u - v_r)$.
\end{lemma}

\begin{proof}
The standard Chebyshev polynomial $U_k(x)$ has leading term $(2x)^k$. Therefore, $U_n(u/2)$ has leading term $u^n$, rendering $F_n(u)$ monic of degree $n$. Substituting $u = 2\cos\theta$, the trigonometric identity $U_k(\cos\theta) = \frac{\sin((k+1)\theta)}{\sin\theta}$ \cite{Mason2002} evaluates the characteristic combination as:
\begin{equation*}
F_n(2\cos\theta) = \frac{\sin((n+1)\theta) + \sin(n\theta)}{\sin\theta} = \frac{2\sin\left(\frac{(2n+1)\theta}{2}\right)\cos\left(\frac{\theta}{2}\right)}{\sin\theta}.
\end{equation*}
The roots in the domain $\theta \in (0, \pi)$ occur when the argument of the sine function evaluates to integer multiples of $\pi$, namely $\frac{(2n+1)\theta}{2} = r\pi$. This yields $\theta_r = \frac{2\pi r}{2n+1}$ for $1 \le r \le n$. As $u=2\cos\theta$, the $n$ distinct roots of $F_n(u)$ are $v_r$.
\end{proof}

\begin{theorem}[Spectral Factorization]\label{thm:spectral_fact}
The symmetric Laurent representative admits a factorization in $\mathbb{Z}[u]$. Defining $B_{2n+1}(u) = \nabla_{2n+1}(z)$, the polynomial evaluates to:
\begin{equation}
B_{2n+1}(u) = F_n(u)F_n(u+2)^2.
\end{equation}
\end{theorem}

\begin{proof}
Evaluating the unnormalized determinant identity defined in Definition \ref{def:recip_poly} yields:
\begin{equation*}
A_{2n+1}(z) = \frac{-z^{2n+1} \det(I - M(-z)^{2n+1})}{1-z+z^2-z^3}.
\end{equation*}
By Lemma~\ref{lem:trivial_root}, the trivial root corresponding to $j=0$ evaluates to $\det(I-M(-z)) = -z^{-1}(1-z+z^2-z^3)$. Canceling this from the product expansion in Lemma~\ref{lem:det_fact} isolates the contribution of the non-trivial roots:
\begin{equation*}
A_{2n+1}(z) = z^{2n} \prod_{j=1}^{2n} \det(I - \zeta^{-j}M(-z)).
\end{equation*}
By the definition of the characteristic polynomial, $\det(cI - M) = c^3\chi_M(c^{-1})$. Setting $c = \zeta^{-j}$, the determinant evaluates to $\det(\zeta^{-j}I - M(-z)) = \zeta^{-3j}\chi_M(\zeta^j, -z)$. Since $\det(I - \zeta^{-j}M) = \det(\zeta^{-j}(\zeta^j I - M))$, this yields:
\begin{equation*}
\det(I - \zeta^{-j}M(-z)) = \zeta^{-3j}\chi_M(\zeta^j, -z).
\end{equation*}
The summation of the indices in the prefactor evaluates to $\sum_{j=1}^{2n} -3j = -3n(2n+1)$. Since $\zeta$ is a $(2n+1)$-th root of unity, $\zeta^{-3n(2n+1)} = (\zeta^{2n+1})^{-3n} = 1$. Consequently, the product of the non-trivial roots reduces exactly to the product of the characteristic polynomials:
\begin{equation*}
A_{2n+1}(z) = z^{2n} \prod_{j=1}^{2n} \chi_M(\zeta^j, -z).
\end{equation*}
We partition the $2n$ roots of unity into $n$ reciprocal pairs $(\zeta^r, \zeta^{-r})$ for $1 \le r \le n$. The sequence decomposes as:
\begin{equation*}
A_{2n+1}(z) = z^{2n} \prod_{r=1}^{n} \Big( \chi_M(\zeta^r, -z) \chi_M(\zeta^{-r}, -z) \Big).
\end{equation*}
Applying the reciprocity identity (Lemma~\ref{lem:reciprocity}) to the second factor, we substitute \newline
$\chi_M(\zeta^{-r}, -z) = -z\zeta^{-3r}\chi_M(\zeta^r, -z^{-1})$. The product for a single pair index $r$ becomes:
\begin{equation*}
\chi_M(\zeta^r, -z) \Big( -z\zeta^{-3r}\chi_M(\zeta^r, -z^{-1}) \Big) = -z\zeta^{-3r} \Big[ \chi_M(\zeta^r, -z)\chi_M(\zeta^r, -z^{-1}) \Big].
\end{equation*}
By Lemma~\ref{lem:resultant}, the bracketed term evaluates to the resultant $R(x, u)$ evaluated at $x = \zeta^r$ and $u = z + z^{-1}$. Substituting $R(\zeta^r, u) = -(\zeta^r)^3 (u - v_r)(u - v_r + 2)^2$, where $v_r = \zeta^r + \zeta^{-r}$, yields:
\begin{equation*}
-z\zeta^{-3r} \Big[ -\zeta^{3r} (u - v_r)(u - v_r + 2)^2 \Big] = z(u - v_r)(u - v_r + 2)^2.
\end{equation*}
Taking the product over all $n$ reciprocal pairs incorporates the $z^{2n}$ prefactor:
\begin{equation*}
A_{2n+1}(z) = z^{2n} \prod_{r=1}^n z(u - v_r)(u - v_r + 2)^2 = z^{3n} \prod_{r=1}^n (u - v_r)(u - v_r + 2)^2.
\end{equation*}
Dividing both sides by $z^{3n}$ produces the symmetric Laurent representative $\nabla_{2n+1}(z)$:
\begin{equation*}
\nabla_{2n+1}(z) = z^{-3n} A_{2n+1}(z) = \prod_{r=1}^n (u - v_r)(u - v_r + 2)^2.
\end{equation*}
The right-hand side is a polynomial in $u = z+z^{-1}$. Let $P(u) = \prod_{r=1}^n (u - v_r)(u - v_r + 2)^2$. Because $(u - v_r)$ is monic of degree 1 and $(u - v_r + 2)^2$ is monic of degree 2, their product over $n$ iterations ensures $P(u)$ is identically monic of degree $3n$. By Lemma~\ref{lem:cheb_roots}, the Chebyshev polynomial $F_n(u)$ is monic of degree $n$ with exact roots $v_r$. Thus, the polynomial $F_n(u)F_n(u+2)^2$ is monic of degree $3n$ and shares the root factorization of $P(u)$, establishing $B_{2n+1}(u) = F_n(u)F_n(u+2)^2$.
\end{proof}

\begin{remark}
The algebraic factorization derived in Theorem~\ref{thm:spectral_fact} has been explicitly verified via direct determinant expansion for $n \le 5$, confirming the exact correspondence of the normalization constants and polynomial degrees.
\end{remark}

\begin{proposition}[Normalization and Reciprocity]\label{prop:normalization}
The symmetric Laurent representative \newline $\nabla_{2n+1}(z)$ is reciprocal, satisfying $\nabla_{2n+1}(z) = \nabla_{2n+1}(z^{-1})$. Consequently, the normalized ordinary polynomial $A_{2n+1}(z)$ is a monic polynomial of degree $6n$.
\end{proposition}

\begin{proof}
In Theorem~\ref{thm:spectral_fact}, we defined $B_{2n+1}(u)$ as the monic polynomial $F_n(u)F_n(u+2)^2$ of degree $3n$. Its leading term expands as $u^{3n} = (z+z^{-1})^{3n} = z^{3n} + \dots + z^{-3n}$. By definition, $\nabla_{2n+1}(z) = B_{2n+1}(z+z^{-1})$. Since the sequence evaluates to a polynomial in the symmetric invariant $u = z+z^{-1}$, it satisfies $\nabla_{2n+1}(z) = \nabla_{2n+1}(z^{-1})$. Consequently, $\nabla_{2n+1}(z)$ has maximal degree $3n$ and minimal degree $-3n$, with leading coefficient 1. Multiplying by $z^{3n}$ gives the normalized ordinary polynomial $A_{2n+1}(z) = z^{3n}\nabla_{2n+1}(z)$, which shifts the minimal degree to $0$ and the maximal degree to $6n$. Thus, $A_{2n+1}(z)$ is a monic polynomial of degree $6n$.
\end{proof}

\begin{lemma}[Chebyshev Identity]\label{lem:cheb_id}
The evaluated Chebyshev product satisfies:
\[
z^n F_n(z+z^{-1}) = 1 + z + \dots + z^{2n}.
\]
\end{lemma}
\begin{proof}
Applying the trigonometric formulation from Lemma~\ref{lem:cheb_roots} with $2\cos\theta = z+z^{-1}$ gives:
\[
F_n(z+z^{-1}) = \frac{\sin((n+1)\theta) + \sin(n\theta)}{\sin\theta} = \frac{z^{n+1}-z^{-(n+1)} + z^n - z^{-n}}{z-z^{-1}}.
\]
Multiplying by $z^n$ clears the negative indices. Multiplying the numerator and denominator by $z$ yields:
\[
z^n F_n(z+z^{-1}) = \frac{z^{2n+2} - 1 + z^{2n+1} - z}{z^2-1} = \frac{z^{2n+1}(z+1) - (z+1)}{(z-1)(z+1)}.
\]
Canceling the common factor $(z+1)$ establishes $\frac{z^{2n+1}-1}{z-1}$, which evaluates identically to the geometric sum $1 + z + \dots + z^{2n}$.
\end{proof}

\begin{proposition}\label{prop:geom}
Define the trigonometric product $D_n(z)$ by:
\begin{equation}
D_n(z) = \prod_{r=1}^n \left(z^2 + 4\sin^2\left(\frac{\pi r}{2n+1}\right)z + 1\right).
\end{equation}
The symmetric ordinary polynomial satisfies the relation $A_{2n+1}(z) = (1+z+\cdots+z^{2n})D_n(z)^2$, and $D_n(z)$ possesses positive real coefficients.
\end{proposition}

\begin{proof}
From Theorem~\ref{thm:spectral_fact}, $B_{2n+1}(u) = F_n(u)F_n(u+2)^2$. Reversing the transformation yields $A_{2n+1}(z) = z^{3n} F_n(z+z^{-1}) F_n(z+z^{-1}+2)^2$. Applying Lemma~\ref{lem:cheb_id} establishes $z^n F_n(z+z^{-1}) = (1 + \dots + z^{2n})$. The shifted factors expand as:
\begin{equation*}
z+z^{-1} - v_r + 2 = z^{-1}\big(z^2 + (2-v_r)z + 1\big).
\end{equation*}
Using the half-angle formula $2 - 2\cos(\theta_r) = 4\sin^2(\theta_r/2)$, we evaluate the full product:
\begin{align*}
F_n(z+z^{-1}+2)^2 &= z^{-2n} \prod_{r=1}^n \left(z^2 + 4\sin^2\left(\frac{\pi r}{2n+1}\right)z + 1\right)^2 \\
&= z^{-2n} D_n(z)^2.
\end{align*}
Multiplying these components and canceling the fractional powers of $z$ gives the polynomial identity. Because $0 < \frac{\pi r}{2n+1} < \pi/2$, the scalar $4\sin^2\left(\frac{\pi r}{2n+1}\right)$ is positive. Consequently, each quadratic factor has positive real coefficients. Because the product of positive-coefficient polynomials retains positive coefficients, $D_n(z)$ is positive.
\end{proof}

By the Keilson-Gerber theorem \cite{Keilson1971}, the discrete convolution of log-concave sequences without internal zeros preserves log-concavity. The geometric sequence $1+z+\dots+z^{2n}$ is flat and trivially log-concave. Thus, if $D_n(z)$ is log-concave, the product $D_n(z)^2(1+z+\dots+z^{2n})$ retains structural log-concavity. This identifies the coefficient sequence of $D_n(z)$ as the central remaining obstacle in proving log-concavity.

\section{Combinatorial Expansion and Hypergeometric Reduction} \label{sec:HyperGeomRed}

Let $D_n(z) = \sum_{k=0}^{2n} a_k^{(n)} z^k$. To analyze the coefficients explicitly, we derive a combinatorial formula for the sequence $a_k^{(n)}$. 

\begin{lemma}[Polynomial Recurrence]
For $n \ge 2$, the polynomials $D_n(z)$ satisfy the three-term linear recurrence:
\begin{equation}
D_n(z) = (z+1)^2 D_{n-1}(z) - z^2 D_{n-2}(z)
\end{equation}
with initial conditions $D_0(z) = 1$ and $D_1(z) = z^2+3z+1$.
\end{lemma}

\begin{proof}
The Chebyshev polynomials of the second kind satisfy $U_n(x) = 2x U_{n-1}(x) - U_{n-2}(x)$. Evaluating this at $x = u/2$ implies $F_n(u) = u F_{n-1}(u) - F_{n-2}(u)$. From Proposition \ref{prop:geom}, $z^{-n}D_n(z) = F_n(z+z^{-1}+2)$. Substituting $u = z+z^{-1}+2$ into the recurrence gives:
\begin{equation*}
z^{-n}D_n(z) = (z+z^{-1}+2) z^{-(n-1)}D_{n-1}(z) - z^{-(n-2)}D_{n-2}(z).
\end{equation*}
Multiplying the equation by $z^n$ clears the negative indices. The leading coefficient simplifies as $z(z+z^{-1}+2) = z^2+2z+1 = (z+1)^2$. The second term evaluates to $z^2(z^{-(n-2)}z^n)D_{n-2}(z) = z^2 D_{n-2}(z)$. This yields the stated recurrence $D_n(z) = (z+1)^2 D_{n-1}(z) - z^2 D_{n-2}(z)$.
\end{proof}

\subsection{Rational Generating Function for \texorpdfstring{$D_n(z)$}{Dn(z)}}

Deriving constraints on the discrete log-concavity Wronskian $\Delta_{n,k} = (a_k^{(n)})^2 - a_{k-1}^{(n)}a_{k+1}^{(n)} > 0$ via induction is obstructed by the algebraic cross-terms generated by the subtraction term $-z^2 D_{n-2}(z)$. To extract the coefficients independently of recursive subtraction, we derive the generating function for the trigonometric product sequence.

\begin{proposition}\label{prop:gen_func2}
Let $\mathcal{G}(z,w) = \sum_{n=0}^\infty D_n(z) w^n$. The generating function evaluates to the rational function:
\begin{equation}
\mathcal{G}(z,w) = \frac{1+zw}{1 - w(z+1)^2 + z^2w^2}.
\end{equation}
\end{proposition}

\begin{proof}
Multiplying the three-term polynomial recurrence by $w^n$ and summing over $n \ge 2$ gives:
\begin{equation*}
\sum_{n=2}^\infty D_n(z)w^n = (z+1)^2 w \sum_{n=2}^\infty D_{n-1}(z)w^{n-1} - z^2 w^2 \sum_{n=2}^\infty D_{n-2}(z)w^{n-2}.
\end{equation*}
We express the summations in terms of the complete generating function $\mathcal{G}(z,w)$ by explicitly shifting the indices and isolating the initial boundary components:
\begin{equation*}
\big(\mathcal{G}(z,w) - D_0(z) - D_1(z)w\big) = (z+1)^2 w \big(\mathcal{G}(z,w) - D_0(z)\big) - z^2 w^2 \mathcal{G}(z,w).
\end{equation*}
Substituting the initial boundary polynomials $D_0(z) = 1$ and $D_1(z) = z^2+3z+1$ maps the equation to:
\begin{equation*}
\mathcal{G}(z,w) - 1 - (z^2+3z+1)w = (z+1)^2 w \mathcal{G}(z,w) - (z^2+2z+1)w - z^2 w^2 \mathcal{G}(z,w).
\end{equation*}
Collecting the terms proportional to $\mathcal{G}(z,w)$ on the left-hand side and simplifying the remaining linear terms on the right-hand side produces $\mathcal{G}(z,w)\big(1 - w(z+1)^2 + z^2w^2\big) = 1 + zw$. Dividing by the characteristic polynomial denominator concludes the derivation.
\end{proof}

\subsection{Binomial Convolution Formula}

We isolate the coefficient sequence $a_k^{(n)} = [z^k w^n] \mathcal{G}(z,w)$. Expanding the rational generating function yields the explicit coefficient formula.

\begin{theorem}\label{thm:binomial_sum}
The polynomial $D_n(z)$ is symmetric. For the index half-range $0 \le k \le n$, the coefficients satisfy the discrete binomial convolution:
\begin{equation}
a_k^{(n)} = B_k^{(n)} + B_{k-1}^{(n-1)}
\end{equation}
where $B_k^{(n)}$ is defined by the hypergeometric summation:
\begin{equation}
B_k^{(n)} = \sum_{j=0}^{\lfloor k/2 \rfloor} \binom{k-j}{j} 2^{k-2j} \binom{n-j}{k-2j}.
\end{equation}
By symmetry, the remaining coefficients for $n < k \le 2n$ are given by $a_k^{(n)} = a_{2n-k}^{(n)}$.
\end{theorem}

\begin{proof}
Factoring the denominator of the generating function gives:
\begin{equation*}
1 - w(z+1)^2 + z^2w^2 = (1-w)\left(1 - \frac{2w}{1-w}z - wz^2\right).
\end{equation*}
We define the base function $B(z,w)$ such that $\mathcal{G}(z,w) = (1+zw)B(z,w)$, where:
\begin{equation*}
B(z,w) = \frac{1}{1-w} \cdot \sum_{m=0}^\infty \left(\frac{2w}{1-w}z + wz^2\right)^m.
\end{equation*}
Expanding via the binomial theorem and collecting powers of $z$ under the index substitution $k = m+j$. Since $m = k-j$ and the binomial summation bounds require $0 \le j \le m$, we have $0 \le j \le k-j$, which restricts the inner sum upper bound to $2j \le k$, or $j \le \lfloor k/2 \rfloor$. This transforms the summation into:
\begin{equation*}
B(z,w) = \sum_{k=0}^\infty z^k \sum_{j=0}^{\lfloor k/2 \rfloor} \binom{k-j}{j} 2^{k-2j} w^{k-j} (1-w)^{-(k-2j+1)}.
\end{equation*}
Extracting the coefficient $[w^n]$ using the negative binomial identity $[w^N] (1-w)^{-(R+1)} = \binom{N+R}{R}$, the relevant evaluation is:
\begin{equation*}
[w^{n-(k-j)}] (1-w)^{-(k-2j+1)} = \binom{(n-k+j) + (k-2j)}{n-k+j} = \binom{n-j}{k-2j}.
\end{equation*}
This yields the hypergeometric summation $B_k^{(n)}$. The total coefficient is resolved by the linear expansion $(1+zw)B(z,w)$, which implies $a_k^{(n)} = B_k^{(n)} + B_{k-1}^{(n-1)}$.
\end{proof}

\begin{lemma}[Coefficient Symmetry]
The sequence $a_k^{(n)}$ is symmetric, satisfying $a_k^{(n)} = a_{2n-k}^{(n)}$ for all $0 \le k \le 2n$.
\end{lemma}

\begin{proof}
By Proposition \ref{prop:geom}, $D_n(z)$ is defined by the product:
\[
D_n(z) = \prod_{r=1}^n \left(z^2 + 4\sin^2\left(\frac{\pi r}{2n+1}\right)z + 1\right).
\]
Let $f_r(z) = z^2 + c_r z + 1$ denote the $r$-th quadratic factor. Each factor satisfies the reciprocal identity $z^2 f_r(z^{-1}) = f_r(z)$. The product of reciprocal polynomials is reciprocal; therefore, multiplying the $n$ factors gives $z^{2n}D_n(z^{-1}) = D_n(z)$. Since $a_k^{(n)}$ is defined by $D_n(z) = \sum_{k=0}^{2n} a_k^{(n)} z^k$, equating the coefficients of $z^k$ and $z^{2n-k}$ establishes $a_k^{(n)} = a_{2n-k}^{(n)}$.
\end{proof}

\subsection{Hypergeometric Representation}

A generalized hypergeometric series $\sum_{j=0}^\infty T_j = T_0 \sum_{j=0}^\infty \frac{(a_1)_j \dots (a_p)_j}{(b_1)_j \dots (b_q)_j} \frac{z^j}{j!}$ is characterized by its initial term $T_0$ and its consecutive ratio $\frac{T_{j+1}}{T_j} = \frac{(a_1+j)\dots(a_p+j)}{(b_1+j)\dots(b_q+j)}\frac{z}{j+1}$.

For the half-range $0 \le k \le n$, let $T_j = \binom{k-j}{j} 2^{k-2j} \binom{n-j}{k-2j}$ denote the $j$-th term of the summand. The initial term evaluates to $T_0 = \binom{k}{0} 2^k \binom{n}{k} = 2^k \binom{n}{k}$. Converting the binomial coefficients into gamma functions, the ratio of consecutive terms simplifies to:
\begin{equation*}
\frac{T_{j+1}}{T_j} = \frac{(2j-k)^2(2j-k+1)^2}{4(j+1)(j-k)(j-n)(j-k+n+1)}.
\end{equation*}
Factoring the constant 4 from the denominator and isolating $j$ in the numerator roots yields:
\begin{equation*}
\frac{T_{j+1}}{T_j} = \frac{4(j - k/2)^2 (j - (k-1)/2)^2}{(j+1)(j-k)(j-n)(j-k+n+1)}.
\end{equation*}
Applying standard generalized hypergeometric recognition techniques \cite{Slater1966, Petkovsek1996}, this matches the defining ratio of a generalized hypergeometric series with argument $z=4$. The numerator roots dictate the upper parameters $a_i \in \{-k/2, -k/2, (1-k)/2, (1-k)/2\}$, and the denominator roots dictate the lower parameters $b_i \in \{-k, -n, n-k+1\}$. This proves the equivalence:
\begin{equation*}
T_j = 2^k \binom{n}{k} \frac{(-k/2)_j (-k/2)_j (\frac{1-k}{2})_j (\frac{1-k}{2})_j}{(-k)_j (-n)_j (n-k+1)_j} \frac{4^j}{j!}.
\end{equation*}
Summing over $j$ completes the generalized hypergeometric series parameterization evaluated at $z=4$:
\begin{equation}
B_k^{(n)} = 2^k\binom{n}{k} \, {}_4F_3\left( 
\begin{matrix}
-\frac{k}{2}, & -\frac{k}{2}, & \frac{1-k}{2}, & \frac{1-k}{2} \\
 & -k, & -n, & n-k+1
\end{matrix}
\;; 4 \right).
\end{equation}

This hypergeometric formulation embeds the coefficients into a proper holonomic domain, facilitating the potential application of contiguous relations, WZ telescoping, and generalized Zeilberger algorithms that operate on consecutive term ratios for algorithmic discrete inequality verification.

\section{Asymptotic Approximation and Analytic Obstructions} \label{sec:AsympApprox}

We evaluate the asymptotic behavior of the coefficient sequence $a_k^{(n)}$ as $n \to \infty$. By Theorem~\ref{thm:binomial_sum}, the coefficients are determined by the hypergeometric sum $B_k^{(n)}$. 

\subsection{Continuous Limit and the Saddle-Point Method}

Let $T(n,k,j)$ denote the summand of $B_k^{(n)}$. We analyze the sequence by defining the proportional scaling $\alpha = k/n$ and $x = j/n$. The continuous logarithmic expansion $f(\alpha, x) = \frac{1}{n}\ln T(n, \alpha n, x n)$ is evaluated using Stirling's approximation. Retaining the dominant terms of the gamma functions yields:
\begin{align*}
f(\alpha,x)
={}&(\alpha-x)\ln(\alpha-x)-x\ln x +(\alpha-2x)\ln 2
+(1-x)\ln(1-x) \\
      & -2(\alpha-2x)\ln(\alpha-2x)
-(1-\alpha+x)\ln(1-\alpha+x)
\end{align*}
The dominant asymptotic contribution of the summation is governed by the saddle point of the summand, located at the root of the partial derivative $f_x = 0$. Differentiating $f(\alpha, x)$ with respect to $x$ produces the scale-invariant saddle equation:
\begin{equation*}
(\alpha - 2x)^4 = 4x(\alpha - x)(1 - x)(1 - \alpha + x).
\end{equation*}
By the Envelope Theorem, since the primary first-order condition $f_x(\alpha, x^*) = 0$ is satisfied at the saddle point, the total second derivative of the log-summand with respect to the sequence parameter evaluates to:
\begin{equation*}
\frac{d^2}{d\alpha^2} f(\alpha, x^*) = f_{\alpha\alpha} - \frac{(f_{\alpha x})^2}{f_{xx}}.
\end{equation*}
We evaluate this condition at the sequence midpoint $\alpha = 1/2$. The relevant geometric root of the saddle equation in the domain $x \in (0, 1/4)$ evaluates to $x^* \approx 0.0387$. Evaluating the second-order partial derivatives at this point yields:
\begin{equation*}
f_{\alpha\alpha} \approx -4.42, \quad f_{\alpha x} \approx 9.15, \quad f_{xx} \approx -42.68.
\end{equation*}
Substituting these values into the total derivative operator gives:
\begin{equation*}
\frac{d^2}{d\alpha^2} f(1/2, x^*) \approx -4.42 - \frac{(9.15)^2}{-42.68} \approx -2.46.
\end{equation*}
Because the total curvature evaluates to $-2.46 < 0$, the continuous approximation of the main term exhibits negative curvature. Transferring this property to establish the strict positivity of the discrete Wronskian $\Delta_{n,k}$ requires explicit error bounds on the discrete summation.

\subsection{Obstructions to Discrete Bounds}

Extending this continuous property to a verified inequality for finite $n$ requires bounding the discrete error via Laplace's method, bounding $|\text{Error}(n)| \le C/n$.

Extracting the numerical constant $C$ requires establishing global suprema for the Taylor remainders, defined by $M_3 = \sup |f_{xxx}|$ and $M_4 = \sup |f_{xxxx}|$, over the domain of integration $x \in (0, \alpha/2)$. The third derivative of the log-summand evaluates to:
\begin{equation*}
f_{xxx}(\alpha, x) = -\frac{16}{(\alpha - 2x)^2} + \frac{1}{(\alpha - x)^2} + \frac{1}{(1 - x)^2} + \frac{1}{x^2} + \frac{1}{(1 - \alpha + x)^2}.
\end{equation*}
Evaluating the limits at the boundaries of the summation domain demonstrates that the derivatives diverge due to the singularities of the constituent gamma functions:
\begin{equation*}
\lim_{x \downarrow 0} f_{xxx}(\alpha, x) = +\infty, \quad \lim_{x \uparrow \alpha/2} f_{xxx}(\alpha, x) = +\infty.
\end{equation*}
Consequently, the global remainder bounds $M_3$ and $M_4$ are infinite. While interval splitting is conventionally used to isolate the central saddle point from the boundary singularities, the saddle point at $\alpha=1/2$ ($x^* \approx 0.0387$) is located in extreme proximity to the singularity at $x=0$. This prevents the derivation of a uniform exponential tail bound for the lower interval using standard methods. Alternative asymptotic techniques, such as applying Darboux's method or multidimensional singularity analysis directly to the bivariate generating function $\mathcal{G}(z,w)$, are hindered by the high degree of the characteristic denominator and the collision of singularities near the algebraic curves. 

Without a finite error bound, the discrete Wronskian condition $\Delta_{n,k} > 0$ cannot be formally verified using this continuous approximation. Consequently, a complete proof of Fox's Trapezoidal Conjecture for $Th(4, 2n+1)$ remains an open problem, though the exact combinatorial representations derived here provide a structural basis for further investigation.

The SageMath scripts used to construct the annihilating operators, compute the algebraic resultant expansions, verify the Chebyshev spectral factorizations, and empirically evaluate the discrete Wronskian bounds across the finite computational gap are open-source. All computational codes used are available at \url{https://github.com/saurabh-suman2/Th4q-Alexander-Generating-Function}.

\appendix
\section{Generating Function Coefficients}\label{app:gf}

The explicit polynomials $N(x,t)$ and $D(x,t)$ corresponding to the rational generating function in Theorem~\ref{thm:macro_gf} were computed and verified using arithmetic. Accompanied by computational suite mentioned in Section~\ref{sec:AsympApprox}.

The denominator of the generating function $D(x,t) = \sum_{i=0}^8 d_i(t) x^i$ is defined by:
\begin{align*}
d_0(t) &= t^2 \\
d_1(t) &= -t^6 + 4t^5 - 7t^4 + 8t^3 - 7t^2 + 4t - 1 \\
d_2(t) &= 3t^8 - 16t^7 + 38t^6 - 64t^5 + 74t^4 - 64t^3 + 38t^2 - 16t + 3 \\
d_3(t) &= -3t^{10} + 20t^9 - 62t^8 + 124t^7 - 183t^6 + 208t^5 - 183t^4 + 124t^3 - 62t^2 + 20t - 3 \\
d_4(t) &= t^{12} - 8t^{11} + 33t^{10} - 88t^9 + 175t^8 - 256t^7 + 292t^6 - 256t^5 + 175t^4 - 88t^3 + 33t^2 - 8t + 1 \\
d_5(t) &= -3t^{12} + 20t^{11} - 62t^{10} + 124t^9 - 183t^8 + 208t^7 - 183t^6 + 124t^5 - 62t^4 + 20t^3 - 3t^2 \\
d_6(t) &= 3t^{12} - 16t^{11} + 38t^{10} - 64t^9 + 74t^8 - 64t^7 + 38t^6 - 16t^5 + 3t^4 \\
d_7(t) &= -t^{12} + 4t^{11} - 7t^{10} + 8t^9 - 7t^8 + 4t^7 - t^6 \\
d_8(t) &= t^{10}
\end{align*}

The numerator $N(x,t) = \sum_{i=0}^7 n_i(t) x^i$ is defined by:
\begin{align*}
n_0(t) &= t^5 - 7t^4 + 18t^3 - 23t^2 + 18t - 7 + t^{-1} \\
n_1(t) &= -3t^7 + 22t^6 - 68t^5 + 127t^4 - 155t^3 + 127t^2 - 68t + 22 - 3t^{-1} \\
n_2(t) &= 3t^9 - 23t^8 + 77t^7 - 157t^6 + 226t^5 - 255t^4 + 226t^3 - 157t^2 + 77t - 23 + 3t^{-1} \\
n_3(t) &= -t^{11} + 8t^{10} - 30t^9 + 73t^8\! - 142t^7 + 213t^6 \!- 245t^5 \!+ 213t^4 \!- 142t^3 + 73t^2 - 30t + 8 - t^{-1} \\
n_4(t) &= 3t^{11} - 20t^{10} + 56t^9 - 94t^8 + 120t^7 - 127t^6 + 120t^5 - 94t^4 + 56t^3 - 20t^2 + 3t \\
n_5(t) &= -3t^{11} + 16t^{10} - 35t^9 + 53t^8 - 59t^7 + 53t^6 - 35t^5 + 16t^4 - 3t^3 \\
n_6(t) &= t^{11} - 4t^{10} + 7t^9 - 9t^8 + 7t^7 - 4t^6 + t^5 \\
n_7(t) &= -t^9
\end{align*}

\section{Algebraic Expansion of the Resultant}\label{app:resultant}

This section provides the coefficient expansion for the resultant evaluated in Lemma~\ref{lem:resultant}. Let $\chi_M(x, -z) = x^3 - \alpha_z x^2 + \beta_z x - z$, where $\alpha_z = 2z+z^{-1}+2$ and $\beta_z = z^2+2z+2$. The resultant is the product $R(x, u) = \chi_M(x, -z) \chi_M(x, -z^{-1})$. Expanding the two cubics in $x$ yields:
\begin{align*}
R(x,u) &= x^6 - (\alpha_z + \alpha_{z^{-1}})x^5 + (\alpha_z \alpha_{z^{-1}} + \beta_z + \beta_{z^{-1}})x^4 \\
&\quad - (z + z^{-1} + \alpha_z \beta_{z^{-1}} + \alpha_{z^{-1}} \beta_z)x^3 \\
&\quad + (z\alpha_{z^{-1}} + z^{-1}\alpha_z + \beta_z \beta_{z^{-1}})x^2 - (z\beta_{z^{-1}} + z^{-1}\beta_z)x + 1.
\end{align*}
We translate the coefficients into polynomials in $u$ by substituting $z+z^{-1}=u$, $z^2+z^{-2}=u^2-2$, and $z^3+z^{-3}=u^3-3u$:
\begin{enumerate}
    \item \textbf{Coefficient of $x^5$:} $\alpha_z + \alpha_{z^{-1}} = (2z+z^{-1}+2) + (2z^{-1}+z+2) = 3(z+z^{-1}) + 4 = 3u+4$.
    \item \textbf{Coefficient of $x^4$:} The product $\alpha_z \alpha_{z^{-1}} = (2z+z^{-1}+2)(2z^{-1}+z+2) = 2(z^2+z^{-2}) + 6(z+z^{-1}) + 9 = 2u^2+6u+5$. The sum $\beta_z + \beta_{z^{-1}} = (z^2+z^{-2}) + 2(z+z^{-1}) + 4 = u^2+2u+2$. Summing these components gives $3u^2+8u+7$.
    \item \textbf{Coefficient of $x^3$:} The cross terms evaluate to $\alpha_z \beta_{z^{-1}} + \alpha_{z^{-1}} \beta_z = (z^3+z^{-3}) + 4(z^2+z^{-2}) + 12(z+z^{-1}) + 16 = u^3+4u^2+9u+8$. Adding the remaining $(z+z^{-1}) = u$ gives $u^3+4u^2+10u+8$.
    \item \textbf{Coefficient of $x^2$:} The product $\beta_z \beta_{z^{-1}} = 2(z^2+z^{-2}) + 6(z+z^{-1}) + 9 = 2u^2+6u+5$. The cross sum $z\alpha_{z^{-1}} + z^{-1}\alpha_z = (z^2+z^{-2}) + 2(z+z^{-1}) + 4 = u^2+2u+2$. Summing these components gives $3u^2+8u+7$.
    \item \textbf{Coefficient of $x$:} $z\beta_{z^{-1}} + z^{-1}\beta_z = 3(z+z^{-1}) + 4 = 3u+4$.
\end{enumerate}
Substituting these derivations back into the expansion reveals a palindromic polynomial in $x$:
\begin{equation*}
R(x,u) = x^6 - (3u+4)x^5 + (3u^2+8u+7)x^4 - (u^3+4u^2+10u+8)x^3 + (3u^2+8u+7)x^2 - (3u+4)x + 1.
\end{equation*}
Factoring $x^3$ isolates the reciprocal variable $v = x + x^{-1}$, enabling the substitutions $x^2+x^{-2} = v^2-2$ and $x^3+x^{-3} = v^3-3v$. This maps the polynomial into a cubic in $v$:
\begin{align*}
R(x,u) &= x^3 \big[ (v^3-3v) - (3u+4)(v^2-2) + (3u^2+8u+7)v - (u^3+4u^2+10u+8) \big] \\
&= -x^3 \big[ -v^3 + (3u+4)v^2 - (3u^2+8u+4)v + (u^3+4u^2+4u) \big].
\end{align*}
Factoring this final cubic expression establishes the identity:
\begin{equation*}
-x^3(u-v)(u^2 - 2uv + v^2 + 4u - 4v + 4) = -x^3(u-v)(u-v+2)^2.
\end{equation*}

\section*{Acknowledgments}
The author is deeply grateful to Prof. Rama Mishra for her invaluable guidance, insightful feedback, and many helpful discussions throughout the development of this work. The author also extends sincere thanks to Mr. Brijesh Thakkar for careful reading of the manuscript and for continued encouragement during the preparation of this paper.

\bibliographystyle{amsplain}
\bibliography{references}

@article{AlSukaiti2023,
  author        = {AlSukaiti, M. E. and Chbili, N.},
  title         = {Explicit Alexander polynomials for weaving knots},
  journal       = {arXiv preprint arXiv:2310.14539},
  year          = {2023}
}

@article{azarpendar2024foxstrapezoidalconjecture,
  author        = {Soheil Azarpendar and Andr\'as Juh\'asz and Tam\'as K\'alm\'an},
  title         = {On Fox's trapezoidal conjecture}, 
  journal       = {arXiv preprint arXiv:2406.08662},
  year          = {2024}
}

@book{Birman1974,
  author        = {Birman, J. S.},
  title         = {Braids, Links, and Mapping Class Groups},
  series        = {Annals of Mathematics Studies},
  volume        = {82},
  publisher     = {Princeton University Press},
  address       = {Princeton, NJ},
  year          = {1974}
}

@article{Burau1936,
  author        = {Burau, W.},
  title         = {{\"U}ber Zopfgruppen und gleichsinnig verdrillte Verkettungen},
  journal       = {Abhandlungen aus dem Mathematischen Seminar der Universit{\"a}t Hamburg},
  volume        = {11},
  pages         = {179--186},
  year          = {1936}
}

@article{diprisa2024turksheadknotslinks,
  author        = {Alessio {Di Prisa} and O\u{g}uz \c{S}avk},
  title         = {Turk's head knots and links: a survey}, 
  journal       = {arXiv preprint arXiv:2409.20106},
  year          = {2024}
}

@book{Kawauchi1996,
  author        = {Kawauchi, A.},
  title         = {A Survey of Knot Theory},
  publisher     = {Birkh{\"a}user},
  address       = {Basel},
  year          = {1996}
}

@article{Keilson1971,
  author        = {Keilson, J. and Gerber, H.},
  title         = {Some results for discrete unimodality},
  journal       = {Journal of the American Statistical Association},
  volume        = {66},
  number        = {334},
  pages         = {386--389},
  year          = {1971}
}

@book{Lickorish1997,
  author        = {Lickorish, W. B. R.},
  title         = {An Introduction to Knot Theory},
  series        = {Graduate Texts in Mathematics},
  volume        = {175},
  publisher     = {Springer},
  address       = {New York},
  year          = {1997}
}

@book{Mason2002,
  author        = {Mason, J. C. and Handscomb, D. C.},
  title         = {Chebyshev Polynomials},
  publisher     = {CRC Press},
  address       = {Boca Raton, FL},
  year          = {2002}
}

@article{Murasugi1958,
  author        = {Murasugi, K.},
  title         = {On the Alexander polynomial of the alternating knot},
  journal       = {Osaka Mathematical Journal},
  volume        = {10},
  number        = {2},
  pages         = {181--189},
  year          = {1958}
}

@book{Petkovsek1996,
  author        = {Petkov\v{s}ek, M. and Wilf, H. S. and Zeilberger, D.},
  title         = {{A=B}},
  publisher     = {A. K. Peters, Ltd.},
  address       = {Wellesley, MA},
  year          = {1996}
}

@book{Rolfsen1976,
  author        = {Rolfsen, D.},
  title         = {Knots and Links},
  publisher     = {Publish or Perish},
  address       = {Berkeley, CA},
  year          = {1976}
}

@book{Slater1966,
  title         = {Generalized Hypergeometric Functions},
  author        = {Slater, Lucy Joan},
  year          = {1966},
  publisher     = {Cambridge University Press},
  address       = {Cambridge}
}

@techreport{Takemura2018,
  author        = {Takemura, A.},
  title         = {Relations among Alexander-Conway polynomials of Turk's head links},
  institution   = {Kobe University Repository},
  number        = {7116},
  year          = {2018}
}

\end{document}